\documentclass[12pt,twoside]{amsart}

\usepackage{amssymb}
\usepackage{graphicx}
\usepackage[pdfauthor={David Cook II, Uwe Nagel},
            pdftitle={Monomial almost complete intersections in codimension three},
            pdfsubject={Commutative algebra and enumerative combinatorics},
            pdfkeywords={Monomial ideals, weak Lefschetz property, determinants, permanents, lozenge tilings, non-intersecting lattice paths, perfect matchings},
            pdfproducer={LaTeX with hyperref},
            pdfcreator={latex->dvips->ps2pdf},
            pdfpagemode=UseOutlines,
            bookmarksopen=false,
            letterpaper,
            bookmarksnumbered=true,
            pdfborder={0 0 0},
            colorlinks=false]{hyperref}
\usepackage[margin=0.95in]{geometry}

\newcommand{\dntri}{\bigtriangledown}
\newcommand{\uptri}{\triangle}
\newcommand{\fa}{\mathfrak{a}}

\newcommand{\mo}{\mathfrak{o}}
\newcommand{\SO}{\mathcal{O}}
\newcommand{\PP}{\mathbb{P}}
\newcommand{\ZZ}{\mathbb{Z}}

\DeclareMathOperator{\charf}{char}

\DeclareMathOperator{\per}{perm}

\DeclareMathOperator{\reg}{reg}
\DeclareMathOperator{\syz}{syz}

\DeclareMathOperator{\soc}{soc}

\newcommand{\st}{\; \mid \;}
\def\urltilda{\kern -.15em\lower .7ex\hbox{\~{}}\kern .04em}

\numberwithin{figure}{section}
\numberwithin{equation}{section}

\newtheorem{theorem}{Theorem}[section]
\newtheorem{lemma}[theorem]{Lemma}
\newtheorem{proposition}[theorem]{Proposition}
\newtheorem{corollary}[theorem]{Corollary}
\newtheorem{conjecture}[theorem]{Conjecture}

\theoremstyle{definition}

\newtheorem{remark}[theorem]{Remark}
\newtheorem{example}[theorem]{Example}

\begin{document}

\title[Syzygy bundles and the weak Lefschetz property]{Syzygy bundles and the weak Lefschetz property of almost complete intersections}

\author[D.\ Cook II]{David Cook II${}^{\star}$}
\address{Department of Mathematics \& Computer Science, Eastern Illinois University, Charleston, IL 46616}
\email{\href{mailto:dwcook@eiu.edu}{dwcook@eiu.edu}}

\author[U.\ Nagel]{Uwe Nagel}
\address{Department of Mathematics, University of Kentucky, 715 Patterson Office Tower, Lexington, KY 40506-0027}
\email{\href{mailto:uwe.nagel@uky.edu}{uwe.nagel@uky.edu}}

\thanks{
    Part of the work for this paper was done while the authors were partially supported by the National Security Agency
    under Grant Number H98230-09-1-0032.
    The second author was also partially supported by the National Security Agency under Grant Number H98230-12-1-0247
    and by the Simons Foundation under grants \#208869 and \#317096.\\
    \indent ${}^{\star}$ Corresponding author.}


\begin{abstract}
Deciding the presence of the weak Lefschetz property often is a challenging problem. Continuing studies in \cite{BK, CN-IJM, MMN-2011}, in this work an in-depth study is carried out in the case of Artinian monomial ideals with four generators in three variables. We use a connection to lozenge tilings to describe semistability of the syzygy bundle of such an ideal, to determine its generic splitting type, and to decide the presence of the weak Lefschetz property.  We provide results in both characteristic zero and positive
    characteristic.
\end{abstract}

\maketitle

\section{Introduction} \label{sec:intro}

The \emph{weak Lefschetz property} for a standard graded Artinian algebra $A$ over a field $K$ is a natural property. It says that there is a linear form $\ell \in A$ such that the multiplication map $\times \ell : [A]_i \rightarrow [A]_{i+1}$ has maximal rank for all $i$ (i.e., it is injective or surjective).
Its presence implies, for example, restrictions on the Hilbert function and graded Betti numbers of the algebra (see \cite{HMNW, MN}). Recent studies have connected the weak Lefschetz property to many other questions (see, e.g., \cite{BMMNZ,BK-p,GIV, MMO, MMMNW, MN-lower, NS1,St-faces}). Thus, a great variety of tools from representation theory, topology, vector bundle theory, hyperplane arrangements, plane partitions, splines, differential geometry, among others has been used to decide the presence of the weak Lefschetz property (see, e.g.,  \cite{BMMNZ2, BK, CGJL, HSS, HMMNWW, KRV, KV, LZ, MMN-2012,MN-survey, Stanley-1980}). An important aspect has also been the role of the characteristic of $K$.

Any Artinian quotient of a polynomial ring in at most two variables has the weak Lefschetz property regardless of the characteristic of $K$ (see \cite{MZ0} and \cite[Proposition 2.7]{CN-small-type}). This is far from true for quotients of rings with three or more variables.  Here we consider quotients $R/I$, where $R = K[x,y,z]$ and $I$ is a monomial ideal containing a power of $x, y$, and $z$. If $I$ has only three generators, then $R/I$ has the weak Lefschetz property, provided the base field has characteristic zero (see \cite{Stanley-1980, ikeda, Wa, BTK}). We focus on the case, where $I$ has four minimal generators, extending previous work in \cite{BK, CN-IJM, MMN-2011}. To this end we use a combinatorial approach developed in \cite{CN-resolutions, CN-small-type} that  involves lozenge tilings, perfect matchings, and families of non-intersecting lattice paths.  Some of our results have already been used in \cite{MMMNW}.

In Section~\ref{sec:trireg}, we recall the connection between  monomial ideals in three variables and so-called  triangular regions.  We use it to establish
sufficient and necessary conditions for a balanced triangular subregion to be tileable (see Corollary~\ref{cor:pp-tileable}).  In Section~\ref{sec:alg},
we show that the tileability of a triangular subregion $T_d (I)$ is related to the semistability of the syzygy bundle of
the ideal $I$ (see Theorem~\ref{thm:tileable-semistable}).  We further recall the relation between  lozenge tilings of
triangular regions and the weak Lefschetz property. All the results up to this point are true for arbitrary Artinian  monomial ideals of $R$.
  In Section~\ref{sec:amaci} we consider exclusively  Artinian monomial ideals  with four minimal generators. Our results on the weak Lefschetz property of $R/I$ are summarized in
Theorem~\ref{thm:amaci-wlp}. In particular, they provide further evidence for a conjecture in \cite{MMN-2011}, which
concerns the case where $R/I$ is a level algebra. Furthermore, we determine the generic splitting type of the syzygy bundle of
$I$ in all cases but one (see Propositions~\ref{pro:st-nss} and \ref{pro:split-type-semist}). In the remaining case we
show that determining the generic splitting type is equivalent to deciding whether $R/I$ has the weak Lefschetz property
(see Theorem~\ref{thm:equiv}). This result is independent of the  characteristic.

\section{Triangular regions}\label{sec:trireg}

Besides introducing notation, we recall needed facts from the combinatorial approach to Lefschetz properties developed in \cite{CN-resolutions, CN-small-type}. We also establish a new criterion for tileability by lozenges.

Let $R = K[x,y,z]$ be a standard graded polynomial ring over a field $K$, i.e., $\deg{x} = \deg{y} = \deg{z} = 1$.
Unless specified otherwise, $K$ is always an arbitrary field. All $R$-modules in this paper are assumed to be finitely generated and graded. Let $A = R/I = \oplus_{j \ge 0} [A]_j$ be a graded quotient of $R$.
The \emph{Hilbert function} of $A$ is the function $h_A: \ZZ \to \ZZ$ given by $h_A(j) = \dim_K [A]_j$. The
\emph{socle} of $A$, denoted $\soc{A}$, is the annihilator of $\mathfrak{m} = (x, y, z)$, the homogeneous
maximal ideal of $R$, that is, $\soc{A} = \{a \in A \st a \cdot \mathfrak{m} = 0\}$.

\subsection{Triangular regions}\label{sub:tri}~

Let $I$ be a monomial ideal of $R$. As $R/I$ is standard graded, the monomials of $R$ of degree $d \in \ZZ$
that are \emph{not} in $I$ form a $K$-basis of $[R/I]_d$.

Let $d \geq 1$ be an integer. Consider an equilateral triangle of side length $d$ that is composed of $\binom{d}{2}$
downward-pointing ($\dntri$) and $\binom{d+1}{2}$ upward-pointing ($\uptri$) equilateral unit triangles. We label the
downward- and upward-pointing unit triangles by the monomials in $[R]_{d-2}$ and $[R]_{d-1}$, respectively, as
follows: place $x^{d-1}$ at the top, $y^{d-1}$ at the bottom-left, and $z^{d-1}$ at the bottom-right, and continue
labeling such that, for each pair of an upward- and a downward-pointing triangle that share an edge, the label of the
upward-pointing triangle is obtained from the label of the downward-pointing triangle by multiplying with a variable.
The resulting labeled triangular region is the \emph{triangular region (of $R$) in degree $d$}
and is denoted $\mathcal{T}_d$. See Figure~\ref{fig:triregion-R}(i) for an illustration.

\begin{figure}[!ht]
    \begin{minipage}[b]{0.48\linewidth}
        \centering
        \includegraphics[scale=1]{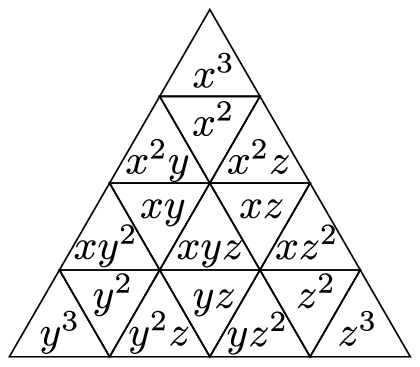}\\
        \emph{(i) $\mathcal{T}_4$}
    \end{minipage}
    \begin{minipage}[b]{0.48\linewidth}
        \centering
        \includegraphics[scale=1]{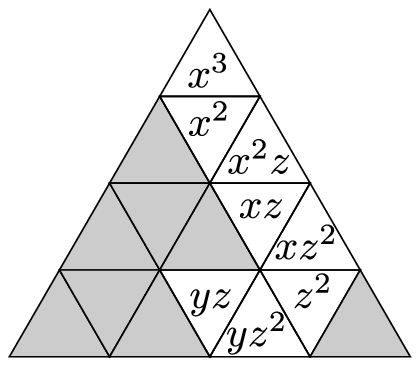}\\
        \emph{(ii) $T_4(xy, y^2, z^3)$}
    \end{minipage}
    \caption{A triangular region with respect to $R$ and with respect to $R/I$.}
    \label{fig:triregion-R}
\end{figure}

Throughout this manuscript we order the monomials of $R$ with the \emph{graded reverse-lexicogra\-phic order}, that is,
$x^a y^b z^c > x^p y^q z^r$ if either $a+b+c > p+q+r$ or $a+b+c = p+q+r$ and the \emph{last} non-zero entry in
$(a-p, b-q, c-r)$ is \emph{negative}. For example, in degree $3$,
\[
    x^3 > x^2y > xy^2 > y^3 > x^2z > xyz > y^2z > xz^2 > yz^2 > z^3.
\]
Thus in $\mathcal{T}_4$, see Figure~\ref{fig:triregion-R}(iii), the upward-pointing triangles are ordered starting at
the top and moving down-left in lines parallel to the upper-left edge.

We generalize this construction to quotients by monomial ideals. Let $I$ be a monomial ideal of $R$. The
\emph{triangular region (of $R/I$) in degree $d$}, denoted by $T_d(I)$, is the part of $\mathcal{T}_d$ that is obtained
after removing the triangles labeled by monomials in $I$. Note that the labels of the downward- and
upward-pointing triangles in $T_d(I)$ form $K$-bases of $[R/I]_{d-2}$ and $[R/I]_{d-1}$, respectively. It is 
more convenient to illustrate such regions with the removed triangles darkly shaded instead of being removed.
See Figure~\ref{fig:triregion-R}(ii) for an example.

Notice that the regions missing from $\mathcal{T}_d$ in $T_d(I)$ can be viewed as a union of (possibly overlapping)
upward-pointing triangles of various side lengths that include the upward- and downward-pointing triangles inside them.
Each of these upward-pointing triangles corresponds to a minimal generator of $I$ that has, necessarily, degree at most
$d-1$. We can alternatively construct $T_d(I)$ from $\mathcal{T}_d$ by removing, for each minimal generator $x^a y^b
z^c$ of $I$ of degree at most $d-1$, the \emph{puncture associated to $x^a y^b z^c$} which is an upward-pointing
equilateral triangle of side length $d-(a+b+c)$ located $a$ triangles from the bottom, $b$ triangles from the
upper-right edge, and $c$ triangles from the upper-left edge. See Figure~\ref{fig:triregion-punctures} for an example.
We call $d-(a+b+c)$ the \emph{side length of the puncture associated to $x^a y^b z^c$}, regardless of possible overlaps
with other punctures in $T_d (I)$.

\begin{figure}[!ht]
    \begin{minipage}[b]{0.48\linewidth}
        \centering
        \includegraphics[scale=1]{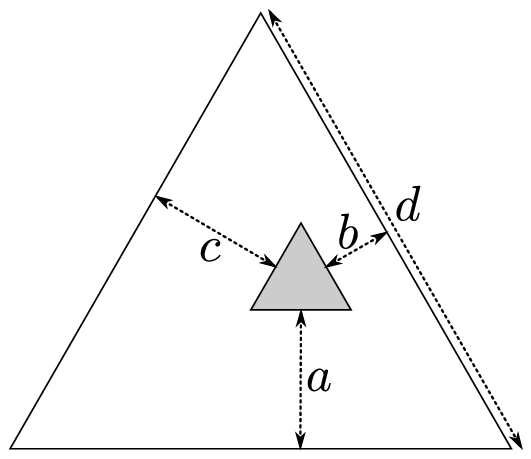}\\
        \emph{(i) $T_{d}(x^a y^b z^c)$}
    \end{minipage}
    \begin{minipage}[b]{0.48\linewidth}
        \centering
        \includegraphics[scale=1]{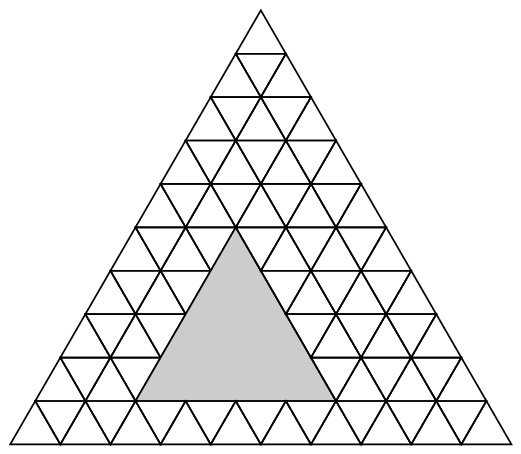}\\
        \emph{(ii) $T_{10}(xy^3z^2)$}
    \end{minipage}
    \caption{$T_d(I)$ as constructed by removing punctures.}
    \label{fig:triregion-punctures}
\end{figure}

We say that two punctures \emph{overlap} if they share at least an edge. Two punctures are said to be \emph{touching}
if they share precisely a vertex.

\subsection{Tilings with lozenges}\label{sub:tiling}~

A \emph{lozenge} is a union of two unit equilateral triangles glued together along a shared edge, i.e., a rhombus with
unit side lengths and angles of $60^{\circ}$ and $120^{\circ}$. Lozenges are also called calissons and diamonds in the
literature.  See Figure~\ref{fig:triregion-intro}.

\begin{figure}[!ht]
    \begin{minipage}[b]{0.48\linewidth}
        \centering
        \includegraphics[scale=1]{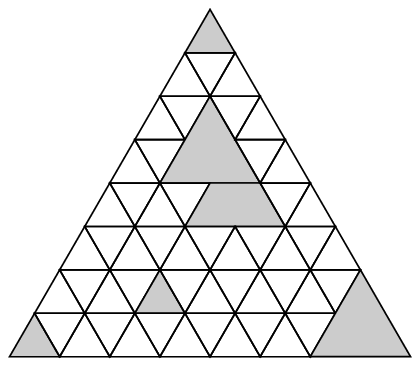}
    \end{minipage}
    \begin{minipage}[b]{0.48\linewidth}
        \centering
        \includegraphics[scale=1]{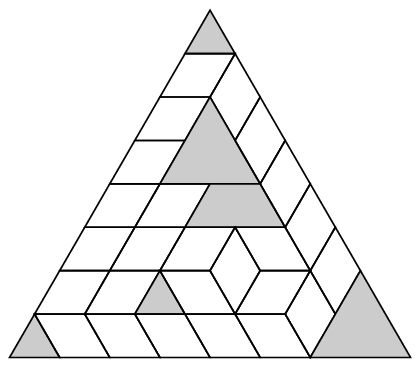}
    \end{minipage}
    \caption{A triangular region $T \subset \mathcal{T}_8$ together with one of its $13$ tilings.}
    \label{fig:triregion-intro}
\end{figure}

Fix a positive integer $d$ and consider the triangular region $\mathcal{T}_d$ as a union of unit triangles. Thus a \emph{subregion}
$T \subset \mathcal{T}_d$ is a subset of such triangles. We retain their labels. As above, we say that
a subregion $T$ is \emph{$\dntri$-heavy}, \emph{$\uptri$-heavy}, or \emph{balanced} if there are more downward pointing
than upward pointing triangles or less, or if their numbers are the same, respectively. A subregion is \emph{tileable}
if either it is empty or there exists a tiling of the region by lozenges such that every triangle is part of exactly one
lozenge.  Since a lozenge in $\mathcal{T}_d$ is the union of a downward-pointing and an upward-pointing triangle, and every
triangle is part of exactly one lozenge, a tileable subregion is necessarily balanced.

Let $T \subset \mathcal{T}_d$ be any subregion. Given a monomial $x^a y^b z^c$ with degree less than $d$, the
\emph{monomial subregion} of $T$ associated to $x^a y^b z^c$ is the part of $T$ contained in the triangle $a$ units from
the bottom edge, $b$ units from the upper-right edge, and $c$ units from the upper-left edge. In other words, this
monomial subregion consists of the triangles that are in $T$ and the puncture associated to the monomial $x^a y^b z^c$.
See Figure~\ref{fig:triregion-subregion} for an example.

\begin{figure}[!ht]
    \includegraphics[scale=1]{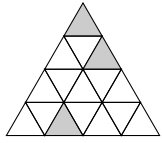}
    \caption{The monomial subregion of $T_{8}(x^7, y^7, z^6, x y^4 z^2, x^3 y z^2, x^4 y z)$
        (see Figure~\ref{fig:triregion-intro}) associated to $x y^2 z$.}
    \label{fig:triregion-subregion}
\end{figure}

Replacing a tileable monomial subregion by a puncture of the same size does not alter tileability.

\begin{lemma}{\cite[Lemma 2.1]{CN-resolutions}}\label{lem:replace-tileable}
    Let $T \subset \mathcal{T}_d$ be any subregion.  If a monomial subregion $U$ of $T$  is tileable,
    then $T$ is tileable if and only if $T \setminus U$ is tileable.

    Moreover, each tiling of $T$ is obtained by combining a tiling of $T \setminus U$ and a tiling of $U$.
\end{lemma}

Let $U \subset \mathcal{T}_d$ be a monomial subregion, and let $T, T' \subset \mathcal{T}_d$ be any subregions such that
$T \setminus U = T' \setminus U$. If $T \cap U$ and $T' \cap U$ are both tileable, then $T$ is tileable if and only if
$T'$ is, by Lemma \ref{lem:replace-tileable}. In other words, replacing a tileable monomial subregion of a triangular
region by a tileable monomial subregion of the same size does not affect tileability.

\begin{theorem}{\cite[Theorem 2.2]{CN-resolutions}}\label{thm:tileable}
    Let $T = T_d(I)$ be a balanced triangular region, where $I \subset R$ is any monomial ideal.  Then $T$ is tileable if and only if
    $T$ has no $\dntri$-heavy monomial subregions.
\end{theorem}

Let $I$ be a monomial ideal of $R$ whose punctures in $\mathcal{T}_d$ (corresponding to the minimal generators of $I$ having degree less than $d$)
have side lengths that sum to $m$. Then we define the \emph{over-puncturing coefficient} of $I$ in degree $d$ to be $\mo_d (I) = m - d$.
If $\mo_d (I) < 0$, $\mo_d (I) = 0$, or $\mo_d (I) > 0$, then we call $I$ \emph{under-punctured}, \emph{perfectly-punctured}, or
\emph{over-punctured} in degree $d$, respectively.

Let now $T = T_d(I)$ be a triangular region with punctures whose side lengths sum to $m$.  Then we define similarly the \emph{over-puncturing coefficient}
of $T$ to be $\mo_d (T) = m - d$.  If $\mo_d (T) < 0$, $\mo_d (T) = 0$, or $\mo_d (T) > 0$, then we call $T$ \emph{under-punctured}, \emph{perfectly-punctured},
or \emph{over-punctured}, respectively.

Observe that different monomial ideals can determine the same triangular region of $\mathcal{T}_d$. Consider, for example,
$I_1 = (x^5, y^5, z^5, xyz^2, xy^2z, x^2yz)$ and $I_2 = (x^5, y^5, z^5, xyz)$. Then $T_6 (I_1) = T_6 (I_2)$, and $\mo_6(I_1) = 3$ but 
$\mo_6(I_2) = 0$.
However, given a triangular region $T = T_d (I)$, there is a unique largest ideal $J$ that is generated by monomials whose
degrees are bounded above by $d-1$ and that satisfies $T = T_d (J)$. We call $J(T)$ the \emph{monomial ideal of the
triangular region $T$}. Note that $\mo_d (T) = \mo_d (J(T)) \leq \mo_d (I)$, and equality is true if and only if the ideals $I$ and
$J(T)$ are the same in all degrees less than $d$.

\begin{remark}
     \label{rem:puncture coeff}
If a monomial subregion $T$  of  $\mathcal{T}_d$   has no overlapping punctures, then  $\mo_d (T)$ is equal to the number of downwards-pointing unit triangles in $T$ minus the number of upward-pointing unit triangles in $T$.
\end{remark}

Perfectly-punctured regions admit a numerical tileability criterion.

\begin{corollary} \label{cor:pp-tileable}
    Let $T = T_d(I)$ be a triangular region.  Then any two of the following conditions imply the third:
    \begin{enumerate}
        \item $T$ is perfectly-punctured;
        \item $T$ has no over-punctured monomial subregions; and
        \item $T$ is tileable.
    \end{enumerate}
\end{corollary}
\begin{proof}
Suppose $T$ is tileable. Then $T$ has no $\dntri$-heavy monomial subregions by Theorem \ref{thm:tileable}.  Thus,  every monomial
    subregion of $T$ is not over-punctured if and only if no punctures of $T$ overlap. Hence (ii) implies (i) by Remark \ref{rem:puncture coeff} because $T$ is balanced. For the converse it is enough to show: If some punctures of $T$ overlap, then $T$ is over-punctured. Indeed, if no punctures overlap, then $T$ is perfectly punctured because $T$ is balanced. So assume two punctures of $T$ overlap. Then the smallest monomial subregion $U$ of $T$ containing these two punctures does not overlap with any other puncture of $T$ and is uniquely tileable. Hence $T \setminus U$ is tileable by Lemma \ref{lem:replace-tileable}, and thus $0 \leq \mo_d (T \setminus U) < \mo_d (T)$, as desired.

If $T$ is non-tileable, then $T$ has a $\dntri$-heavy monomial subregion.   Since every $\dntri$-heavy monomial subregion is
    also over-punctured, it follows that $T$ has an over-punctured monomial subregion.
\end{proof}

Any subregion $T \subset \mathcal{T}_d$ can be associated to a bipartite planar graph $G$ that is an induced subgraph
of a honeycomb graph (see \cite{CN-resolutions}).  We are interested in the  bi-adjacency matrix $Z(T)$ of $G$. This is a zero-one matrix whose determinant enumerates signed lozenge tilings (see \cite[Theorem 3.5]{CN-resolutions}). If $T = T_d(I)$ for some monomial ideal $I$, then $Z(T)$ admits an alternative description. Indeed, consider the
    multiplication map $\times(x+y+z): [R/I]_{d-2} \rightarrow [R/I]_{d-1}$. Let $M(d)$ be the matrix to this linear map
    with respect to the monomial bases of $[R/I]_{d-2}$ and $[R/I]_{d-1}$ in reverse-lexicographic order. Then the
    transpose of $M(d)$ is the bi-adjacency matrix $Z(T_d (I))$ (see \cite[Proposition 4.5]{CN-small-type}). Here we need only a special case of these results.

\begin{proposition}
   \label{prop:pm-det}
Assume $T = T_d (I) \subset \mathcal{T}_d$ is a non-empty balanced subregion.  If $\det Z(T) \in \ZZ$ is not zero, then $T$ is tileable.
\end{proposition}

\begin{proof}
Balancedness of $T$ is equivalent to $\dim_K [R/I]_{d-2} = \dim_K [R/I]_{d-1}$. It follows that $Z(T)$ is a square matrix by \cite[Proposition 4.5]{CN-small-type}. Now \cite[Theorem 3.5]{CN-resolutions} gives the assertion.
\end{proof}

We conclude this section with a criterion that guarantees non-vanishing of $\det Z(T)$. To this end
we recursively define a puncture of $T \subset \mathcal{T}_d$ to be
a \emph{non-floating} puncture if it  touches the boundary of $ \mathcal{T}_d$ or if it overlaps or touches a non-floating puncture of $T$. Otherwise we call a puncture
a \emph{floating} puncture. For example, the region $T$ in Figure \ref{fig:triregion-intro} has three non-floating punctures (in the corners) and three floating punctures, two of them are overlapping and have side length two.

\begin{proposition}{\cite[Corollary 4.7]{CN-resolutions}}
     \label{prop:same-sign}
    Let $T$ be a tileable triangular region, and suppose all floating punctures of $T$ have an even side length. Then
    $\per{Z(T)} = |\det{Z(T)}| \neq 0$.
\end{proposition}

\section{Combinatorial interpretations of some algebraic properties}\label{sec:alg}

In this section, we use the connection to triangular regions to reinterpret some algebraic properties.

\subsection{Stability of syzygy bundles} \label{sub:syz}~

Throughout  this subsection, we assume the characteristic of $K$ is zero.

Let $I$ be an Artinian ideal of $S = K[x_1, \ldots, x_n]$ that is minimally generated by forms $f_1, \ldots, f_m$.
The \emph{syzygy module} of $I$ is the graded module $\syz{I}$ that fits into the exact sequence
\[
    0 \rightarrow \syz{I} \rightarrow \bigoplus_{i=1}^{m}S(-\deg f_i) \rightarrow I \rightarrow 0.
\]
Its sheafification $\widetilde\syz{I}$ is a vector bundle on $\PP^{n-1}$, called the \emph{syzygy bundle} of $I$. It has rank $m-1$.

Semistability is an important property of a vector bundle.  Let $E$ be a vector bundle on projective space.  The \emph{slope}
of $E$ is defined as $\mu(E) := \frac{c_1(E)}{rk(E)}$.  Furthermore, $E$ is said to be \emph{semistable}
if the inequality $\mu(F) \leq \mu(E)$ holds for every coherent subsheaf $F \subset E$.  If the inequality is always strict, then
$E$ is said to be \emph{stable}.

Brenner established a beautiful characterization of the semistability of syzygy bundles to monomial ideals.
Since we only consider monomial ideals in this work, the following may be taken as the definition of (semi)stability herein.

\begin{theorem}{\cite[Proposition~2.2 \& Corollary~6.4]{Br}} \label{thm:stable-syz}
    Let $I$ be an Artinian ideal in $K[x_1, \ldots, x_n]$ that is minimally generated by monomials $g_1, \ldots, g_m$, where $K$ is a field of
    characteristic zero. Then $I$ has a semistable syzygy bundle if and only if, for every proper subset $J$ of $\{1, \ldots, m\}$ with at
    least two elements, the inequality
    \[
        \frac{d_J - \displaystyle\sum_{j \in J} \deg g_j}{|J|-1} \leq \frac{-\displaystyle\sum_{i=1}^{m} \deg g_i}{m-1}
    \]
    holds, where $d_J$ is the degree of the greatest common divisor of the $g_j$ with $j \in J$.  Further, $I$ has a stable
    syzygy bundle if and only if the above inequality is always strict.
\end{theorem}

We use Brenner's criterion to rephrase (semi)stability in the case of a monomial ideal of $K[x,y,z]$ in terms of the
over-puncturing coefficients of ideals.  Note, in particular, that $\mo_d (I)= \sum_{i=1}^{m} (d - \deg g_i) - d$.

\begin{corollary} \label{cor:T-stable}
    Let $I$ be an Artinian ideal in $R = K[x,y,z]$ that is minimally generated by monomials $g_1, \ldots, g_m$ of degree at most $d$.
    For every proper subset $J$ of $\{1, \ldots, m\}$ with at least two elements, let $I_J$ be the monomial ideal that is
    generated by $\{ g_j / g_J \st j \in J\}$, where $g_J = \gcd\{g_j \st j \in J \}$ has degree $d_J$.

    Then $I$ has a semistable syzygy bundle if and only if, for every proper subset $J$ of $\{1, \ldots, m\}$ with at
    least two elements, the inequality
    \[
        \frac{\mo_{d-d_J} (I_J)}{|J| - 1} \leq \frac{\mo_{d} (I)}{m-1}
    \]
    holds.  Furthermore, $I$ has a stable syzygy bundle if and only if the above inequality is always strict.
\end{corollary}

\begin{proof}
  Since $\mo_{d-d_J} (I_J)= d (|J| -1) + d_J -  \sum_{j \in J} \deg g_i$,   this follows immediately from Theorem~\ref{thm:stable-syz}.
\end{proof}

In order to apply this result we slightly extend the concept of a triangular region $T_d (I)$. Label the vertices in $\mathcal{T}_d$ by monomials of degree $d$ such that the label of
each unit triangle is the greatest common divisor of its vertex labels.  Then a minimal monomial generator of $I$ with
degree $d$ corresponds to a vertex of $\mathcal{T}_d$ that is removed in $T_d (I)$. We consider this removed vertex
as a puncture of side length zero. Observe that this is in line with our general definition of the side length of a puncture.

Using Corollary~\ref{cor:pp-tileable}, we see that semistability is strongly related to tileability of a region.

\begin{theorem} \label{thm:tileable-semistable}
    Let $I$ be an Artinian ideal in $R = K[x,y,z]$ generated by monomials whose degrees are bounded above by $d$,
    and let $T = T_d(I)$.  If $T$ is non-empty, then any two of the following conditions imply the third:
    \begin{enumerate}
        \item $I$ is perfectly-punctured;
        \item $T$ is tileable; and
        \item $\widetilde\syz{I}$ is semistable.
    \end{enumerate}
\end{theorem}

\begin{proof}
Assume $I$ is perfectly punctured, that is, $\mo_d(I) = 0$.  We will show that $T$ is tileable if and only
    if $\widetilde\syz{I}$ is semistable.

    If $T$ is tileable, then $\mo_d(T) = 0$ which implies $J(T) = I$.  Hence no punctures of $T$ overlap.  This
    further implies $I_A = J(T_{I_A})$ for any subset $A$ of the generators of $I$.  Thus $\mo_{d-d_A}(I_A) = \mo_{d-d_A}(T_{I_A}) \leq 0$,
    since no punctures overlap and  every subregion is not over-punctured by Corollary~\ref{cor:pp-tileable}.  Hence,  $\widetilde{\syz} I$ is semistable by Corollary~\ref{cor:T-stable}.

    If $\widetilde{\syz} I$ is semistable, then $\mo_{d-d_A}(I_A) \le 0$ holds for any subset $A$ of the generators of $I$.
    This implies, in particular, that no punctures of $T$ overlap.  Hence $I = J(T)$ and so $\mo_d(T) = \mo_d(I) = 0$.
    Furthermore, since no pair of punctures overlap, having no over-punctured regions is the same as having no $\dntri$-heavy regions (see Remark \ref{rem:puncture coeff}).      Thus, by Corollary~\ref{cor:pp-tileable}, $T$ is tileable.

    Now assume $I$ is not perfectly-punctured, but $T$ is tileable. We have to show that $\widetilde\syz{I}$ is not semistable. Arguing as in the proof of Corollary \ref{cor:pp-tileable}, we conclude that  $T$ is over-punctured and must have overlapping punctures.
    Consider two such overlapping punctures of $T$. Then the smallest monomial subregion $U$ containing these two punctures does
    not overlap with any other puncture of $T$ with positive side length.
    Hence $T' = T \setminus U$ is tileable  and $0 \leq \mo_{T'} < \mo_d (I)$.  If $T'$ is still
    over-punctured, then we repeat the above replacement procedure until we get a perfectly-punctured monomial subregion of $T$.
    Abusing notation slightly, denote this region by $T'$.  Let $J$ be the largest monomial ideal containing $I$ and with generators
    whose degrees are bounded above by $d$   such that $T' = T_d (J)$. Observe  that $\mo_d (J) = \mo_{d}  (T') = 0$.

    Notice that a single replacement step above amounts to replacing the triangular region to an ideal $I'$ by the region
    to the ideal $(I', f)$, where $f$ is a greatest common divisor of the minimal generators of $I'$ that correspond to two overlapping punctures. These generators   have
    degrees less than $d$.

    Assume now that $T'$ is empty.  Then $I$ has two relatively prime minimal generators, say $g_1, g_2$, whose corresponding punctures 
    overlap and are not both contained in a proper monomial subregion of $\mathcal{T}_d$. Since $I$ is Artinian it has $m \ge 3$ minimal 
    generators. Moreover, all minimal  generators of $I$ other than $g_1$ and $g_2$ have degree $d$. It follows that 
    $\mo_d((g_1, g_2)) = \mo_d (I)$. Since $m > 2$, Corollary~\ref{cor:T-stable} shows that $\widetilde{\syz}{I}$ is not semistable.

    It remains to consider the case where $T'$ is not empty, i.e., $J$ is a proper ideal of $R$.  Let $g_1, \ldots, g_m$ and
    $f_1, \ldots, f_n$ be the minimal monomial generators of $I$ and $J$, respectively. Partition the generating set of $I$
    into $F_j = \{ g_i \st g_i \mbox{~divides~} f_j\}$. Notice $f_j = \gcd\{F_j\}$. In particular, $n > 1$ as $J$ is a proper
    Artinian ideal.

    Set $\mo_j = \mo_{d - \deg f_j} ((\frac{F_j}{f_j})) =  \sum_{g \in F_j} (d - \deg{g}) - (d - \deg{f_j})$. Observe $\mo_j \geq 0$ as the subregion of $T_d (I)$ associated
    to  $f_j$ is tileable, hence not under-punctured.  Moreover,
    \begin{equation*}
        \begin{split}
            \mo_d (J)  = \sum_{j = 1}^{n}(d - \deg{f_j}) - d & = \sum_{j=1}^{n} \left( \sum_{g \in F_j} (d - \deg{g}) - \mo_j \right) - d \\
                                                            & = \sum_{j=1}^{n} \sum_{g \in F_j} (d - \deg{g}) - d - \sum_{j=1}^{n} \mo_j \\
                                                            & = \mo_d (I) - \sum_{j=1}^{n} \mo_j.
        \end{split}
    \end{equation*}
    As $\mo_d (J)  = 0$, we conclude that $\mo_d (I) = \sum_{j=1}^{n} \mo_j$ and, in particular, $\mo_d (I) \geq \mo_j$ for each $j$.

    Assume $m \cdot \mo_j < \# F_j \cdot \mo_d (I)$ for all $j$.  Then $m \sum_{j=1}^{n}\mo_j < \mo_d (I) \sum_{j=1}^{n} \# F_j = \mo_d (I) \cdot m$.  But this
    implies $m \cdot \mo_d (I) < m \cdot \mo_d (I)$, which is absurd.  Hence, there is some $k$ such that  $m \cdot \mo_k \geq \# F_k \cdot \mo_d (I)$.
    Since $\mo_d (I)\geq \mo_k$ it follows that $\frac{\mo_k}{\# F_k-1} > \frac{\mo_d (I)}{m-1}$.  Indeed, this is immediate if
    $\mo_d (I) > \mo_k$. If $\mo_d (I) = \mo_k$, then it is also true because $\# F_k  < m$. Now Corollary~\ref{cor:T-stable} gives
    that $\widetilde\syz{I}$ is not semistable.
\end{proof}

We get the following criterion when focusing solely on the triangular region. Recall that $J(T)$ denotes the monomial ideal of
a triangular region $T$ as introduced above Remark~\ref{rem:puncture coeff}.

\begin{corollary}  \label{cor:semistability-by-region}
    Let $I$ be an Artinian ideal in $R = K[x,y,z]$ generated by monomials whose degrees are bounded above by $d$,
    and let $T = T_d(I)$.  Assume $T$ is non-empty and tileable.
    \begin{enumerate}
        \item If $I \neq I + J(T)$, then $\widetilde\syz{I}$ is not semistable.
        \item $\widetilde\syz(I + J(T))$ is semistable if and only if  $T$ is perfectly-punctured.
    \end{enumerate}
\end{corollary}
\begin{proof}
Note  that $I \ne I + J(T)$ implies $\mo_d (I+J(T)) < \mo_d (I)$.
 Since $T$ is balanced, we get $0 \leq \mo_d (T) = \mo_d (J(T)) = \mo_d (I + J(T))$. Hence Theorem~\ref{thm:tileable-semistable} gives our assertions.
    \end{proof}

For stability, we obtain the following result.

\begin{proposition} \label{pro:pp-stable}
    Let $I$ be an Artinian ideal in $R = K[x,y,z]$ generated by monomials whose degrees are bounded above by $d$.
    If $T = T_d(I)$ is non-empty, tileable, and perfectly-punctured, then $\widetilde\syz(I + J(T))$ is stable
    if and only if every proper monomial subregion of $T$ is under-punctured.
\end{proposition}
\begin{proof}
    We may assume $I = I + J(T)$. As $T$ is perfectly-punctured, we have that $\mo_d (I) = \mo_d (T) = 0$.  In particular, no punctures of $T$ overlap.  Using Corollary~\ref{cor:T-stable}, we see that
    $\widetilde\syz{I}$ is stable if and only if $\mo_{d-d_J} (T_{d-d_J}(I_J)) < 0$ for all
    proper subsets $J$ of the set of minimal generators of $I$.  This is equivalent to every proper monomial subregion of $T$ being under-punctured (see Remark~\ref{rem:puncture coeff}).
\end{proof}

By the preceding theorem and proposition, we have an understanding of semistability and stability for
perfectly-punctured triangular regions.  However, when a region is over-punctured and non-tileable  more
information is needed to infer semistability.

\begin{example} \label{exa:stability}
    There are monomial ideals with stable syzygy bundles whose corresponding triangular regions are
    over-punctured and  non-tileable.  See Figure~\ref{fig:nss-examples}(i) for a specific  example.

    \begin{figure}[!ht]
        \begin{minipage}[b]{0.30\linewidth}
            \centering
            \includegraphics[scale=1]{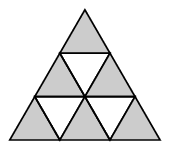}\\
            \emph{(i) $T_3(x^2, y^2, z^2, xy, xz, yz)$}
        \end{minipage}
        \begin{minipage}[b]{0.28\linewidth}
            \centering
            \includegraphics[scale=1]{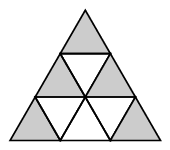}\\
            \emph{(ii) $T_3(x^2, y^2, z^2, xy, xz)$}
        \end{minipage}
        \begin{minipage}[b]{0.40\linewidth}
            \centering
            \includegraphics[scale=1]{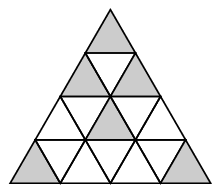}\\
            \emph{(iii) $T_4(x^3, y^3, z^3, xyz, x^2y, x^2z)$}
        \end{minipage}
        \caption{Over-punctured, non-tileable regions and various levels of stability.}
        \label{fig:nss-examples}
    \end{figure}

    Moreover, the ideal $(x^2, y^2, z^2, xy, xz)$ has a semistable, but non-stable  syzygy bundle (the monomial subregion associated
    to $x$ breaks stability), and the ideal $(x^3, y^3, z^3, xyz, x^2y, x^2z)$ has a non-semistable
   syzygy bundle (the monomial subregion associated to $x^2$ breaks semistability).  Both of their triangular regions,
    see Figures~\ref{fig:nss-examples}(ii) and (iii), respectively, are over-punctured and non-tileable.
\end{example}~

\subsection{The weak Lefschetz property}\label{sub:wlp}~

We recall some results that help decide the presence of the weak Lefschetz property. In fact, one needs only check near a ``peak'' of the Hilbert function.

\begin{proposition}{\cite[Proposition 2.3]{CN-small-type}} \label{pro:wlp}
    Let $A \neq 0$ be an Artinian standard graded $K$-algebra, and let $\ell$ be a general linear form.  Suppose $A$ has no non-zero socle elements of degree less than $d-2$ for some integer $d \ge 0$. Then   $A$ has the weak Lefschetz property, provided one of the following conditions is satisfied:

 \begin{itemize}
   \item[(i)] $\times \ell: [A]_{d-2} \rightarrow [A]_{d-1}$ is injective  and
     $\times \ell: [A]_{d-1} \rightarrow [A]_{d}$ is surjective.

   \item[(ii)]  $\times \ell: [A]_{d-2} \rightarrow [A]_{d-1}$ is bijective.
\end{itemize} \end{proposition}

Moreover, for monomial algebras, it is enough to decide whether the sum of the variables is a Lefschetz element.

\begin{proposition}{\cite[Proposition~2.2]{MMN-2011}}
            \label{pro:mono}
Let $A = R/I$ be a monomial Artinian $K$-algebra, where $K$ is an infinite field.
    For any integer $d$, the following conditions are equivalent:
    \begin{enumerate}
        \item The multiplication map $\times L: [A]_{d-1} \to [A]_d$ has maximal rank, where $L \in R$ is a general linear form.
        \item The multiplication map $\times (x + y + z): [A]_{d-1} \to [A]_d$ has maximal rank.
    \end{enumerate}
\end{proposition}

As pointed out above Proposition \ref{prop:pm-det}, for a monomial ideal $I \subset K[x, y, z]$,  the bi-adjacency matrix $Z(T_d (I))$ can be described using  multiplication by $\ell = x+y+z$.  We thus get the following criterion for
the presence of the weak Lefschetz property, where we consider the entries of $Z(T_d (I))$ as elements of the base field $K$.

\begin{corollary}{\cite[Corollary 4.7]{CN-small-type}}\label{cor:wlp-biadj}
    Let $I$ be an Artinian monomial ideal in $R = K[x,y,z]$. Then $R/I$ has the weak Lefschetz property if and only if,
    for each positive integer $d$, the matrix $Z(T_d(I))$ has maximal rank.
\end{corollary}

This can be used to infer the weak Lefschetz property in sufficiently large characteristic from its presence in characteristic zero.

\begin{proposition}{\cite[Proposition 7.9]{CN-small-type}}\label{pro:char-0-to-p}
    Let $R/I$ be any Artinian monomial algebra such that $R/I$ has the weak Lefschetz property in characteristic zero.
    If $I$ contains the powers $x^a, y^b, z^c$, then $R/I$ has the weak Lefschetz property in positive characteristic
    whenever $\charf K > 3^{\frac{1}{2}\binom{\frac{1}{2} (a+b+c) + 2}{2}}$.
\end{proposition}

\section{Artinian monomial almost complete intersections} \label{sec:amaci}

This section
presents an in-depth discussion of  Artinian monomial ideals of $R$ with exactly four minimal generators. They are called
Artinian monomial almost complete intersections. These ideals have been discussed, for example, in \cite{BK}
and~\cite[Section~6]{MMN-2011}. In particular, we will answer some of the questions posed in \cite{MMN-2011}. Besides addressing the  weak Lefschetz property, we discuss the splitting types of the syzygy bundles of these ideals. Particular attention is paid if the characteristic is positive.
Some of our results are used in \cite{MMMNW} for studying ideals with the Rees property.

Each Artinian ideal of $K[x,y,z]$ with exactly four monomial minimal generators is of the form
\[
    I_{a,b,c,\alpha,\beta,\gamma} = (x^a, y^b, z^c, x^\alpha y^\beta z^\gamma),
\]
where $0 \leq \alpha < a$, $0 \leq \beta < b,$ and $0 \leq \gamma < c$, such that at most one of $\alpha$, $\beta$, and
$\gamma$ is zero. If one of $\alpha$, $\beta$, and $\gamma$ is zero, then $R/I_{a,b,c,\alpha,\beta,\gamma}$ has type
two. In this case, the presence of the weak Lefschetz property has already been described in~\cite{CN-small-type}.
Thus, throughout this section we assume that the integers $\alpha, \beta$, and $\gamma$ are all positive;
this forces $R/I_{a,b,c,\alpha,\beta,\gamma}$ to have Cohen-Macaulay type three. More precisely:

\begin{proposition}{\cite[Proposition 6.1]{MMN-2011}} \label{pro:amaci-props}
    Let $I = I_{a,b,c,\alpha,\beta,\gamma}$ be defined as above. Then $R/I$ has three minimal socle generators. They
    have degrees $\alpha + b + c - 3$, $a + \beta + c - 3$, and $a + b + \gamma - 3$.

    In particular, $R/I$ is level if and only if $a - \alpha = b - \beta = c - \gamma$.
\end{proposition}~

\subsection{Presence of the weak Lefschetz property}~\par\label{subsec:aci-wlp}

Brenner made Theorem \ref{thm:stable-syz} more explicit in the situation at hand.

\begin{proposition}{\cite[Corollary~7.3]{Br}} \label{pro:amaci-semistable}
    Let $I = I_{a,b,c,\alpha,\beta,\gamma}$ be defined as above, and suppose $K$ is a field of characteristic zero. Set
    $d = \frac{1}{3}(a+b+c+\alpha + \beta + \gamma)$. Then $I$ has a semistable syzygy bundle if and only if the
    following three conditions are satisfied:
    \begin{enumerate}
        \item $\max\{a, b, c, \alpha + \beta + \gamma\} \leq d$;
        \item $\min\{\alpha + \beta + c, \alpha + b + \gamma, a + \beta + \gamma\} \geq d$; and
        \item $\min\{a+b, a+c, b+c\} \geq d$.
    \end{enumerate}
\end{proposition}

Furthermore, Brenner and Kaid showed that, for almost complete intersections, nonsemistability implies the weak
Lefschetz property in characteristic zero.

\begin{proposition}{\cite[Corollary~3.3]{BK}} \label{pro:amaci-nss-wlp}
    Let $K$ be a field of characteristic zero. Then $I_{a,b,c,\alpha,\beta,\gamma}$ has the weak Lefschetz property if
    its syzygy bundle is not semistable.
\end{proposition}

The conclusion of this result is not necessarily true in positive characteristic.

\begin{example} \label{exa:amaci-nss}
    Let $I = I_{5,5,3,1,1,2}$, and thus $d = 6$. Then the syzygy bundle of $I$ is not semistable as
    $\alpha + \beta + c = 5 < d = 6$. However, the triangular region $T_6(I)$ is balanced and $\det{Z(T_6(I))} = 5$.
    Hence, $I$ does not have the weak Lefschetz property if and only if the characteristic of $K$ is $5$.
\end{example}

The following example illustrates that the assumption on the number of minimal generators cannot be dropped in
Proposition~\ref{pro:amaci-semistable}.

\begin{example} \label{exa:amaci-nss-2}
    Consider the ideal $J = (x^5, y^5, z^5, xy^2z, xyz^2)$ with five minimal generators. Then
    Theorem~\ref{cor:T-stable} gives that the syzygy bundle of $J$ is not semistable. Notice that $T_6(J)$ is
    balanced. However, $\det{Z(T_6(J))} = 0$, and so $R/J$ never has the weak Lefschetz property, regardless of the
    characteristic of $K$.
\end{example}

The number $d$ in Proposition~\ref{pro:amaci-semistable} is not assumed to be an integer. In fact, if it is not, then
the algebra has the weak Lefschetz property.

\begin{proposition}{\cite[Theorem~6.2]{MMN-2011}} \label{pro:amaci-not-3}
    Let $K$ be a field of characteristic zero. Then $I_{a,b,c,\alpha,\beta,\gamma}$ has the weak Lefschetz property if
    $a+b+c+\alpha + \beta + \gamma \not\equiv 0 \pmod{3}$.
\end{proposition}

Again, the conclusion of this result may fail in positive characteristic. Indeed, for the ideal $I_{5,5,3,1,1,2}$ in
Example~\ref{exa:amaci-nss} we get $d = \frac{17}{3}$, but it does not have the weak Lefschetz property in
characteristic $5$.

The following result addresses the weak Lefschetz property in the cases that are left out by
Propositions~\ref{pro:amaci-nss-wlp} and~\ref{pro:amaci-not-3}. Its first part extends \cite[Lemma~7.1]{MMN-2011} from
level to arbitrary monomial almost complete intersections. Observe that balanced triangular regions correspond to an
equality of the Hilbert function in two consecutive degrees, dubbed ``twin-peaks'' in \cite{MMN-2011}.

\begin{proposition} \label{pro:amaci-balanced}
    Let $I = I_{a,b,c,\alpha,\beta,\gamma}$, and assume $d = \frac{1}{3}(a+b+c+\alpha+\beta+\gamma)$ is an integer. If
    the syzygy bundle of $I$ is semistable and $d$ is integer, then $T_d(I)$ is perfectly-punctured and balanced.

    Moreover, in this case $R/I$ has the weak Lefschetz property if and only if $\det{Z(T_d(I))}$ is not zero in $K$.
\end{proposition}
\begin{proof}
    Note that condition (i) in Proposition~\ref{pro:amaci-semistable} says that $T_d(I)$ has punctures of nonnegative
    side lengths $d-a, d-b, d-c$, and $d-(\alpha + \beta + \gamma)$. Furthermore, conditions (ii) and (iii) therein are
    equivalent to the fact that the degree of the least common multiple of any two of the minimal generators of $I$ is
    at least $d$, i.e., the punctures of $T_d(I)$ do not overlap. Using the
    assumption that $d$ is an integer, it follows that $T_d(I)$ is perfectly-punctured, and thus balanced.

    Since the punctures of $T_d(I)$ do not overlap, the punctures of $T_{d-1}(I)$ are not overlapping nor touching.
    Thus we conclude that the degrees of the socle generators of $R/I$ are at
    least $d-2$. Hence, Corollary~\ref{cor:wlp-biadj} and Proposition~\ref{pro:wlp} together give that $R/I$ has the weak Lefschetz property if and only if
    $\det{Z(T_d(I))}$ is not zero in $K$.
\end{proof}

In the situation of Proposition~\ref{pro:amaci-balanced}, the fact that $R/I$ has the weak Lefschetz property implies
that $T_d(I)$ is tileable by Proposition~\ref{prop:pm-det}. Tileability remains true even if $R/I$ fails
to have the weak Lefschetz property.

\begin{proposition} \label{pro:amaci-ss-tileable}
    Let $I = I_{a,b,c,\alpha,\beta,\gamma}$. If $R/I$ fails to have the weak Lefschetz property in characteristic zero,
    then $d = \frac{1}{3}(a+b+c+\alpha+\beta+\gamma)$ is an integer and $T_d(I)$ is tileable.
\end{proposition}
\begin{proof}
    By Propositions~\ref{pro:amaci-nss-wlp} and~\ref{pro:amaci-not-3}, we know that the syzygy bundle of $I$ is
    semistable and $d = \frac{1}{3}(a+b+c+\alpha+\beta+\gamma)$ is an integer. Hence by Proposition~\ref{pro:amaci-balanced},
    $T_d(I)$ is perfectly-punctured. Now we conclude by Theorem~\ref{thm:tileable-semistable}.
\end{proof}

Before we analyze the presence of the weak Lefschetz property, we need to recall a special type of puncture that
has been previously studied by Ciucu, Eisenk\"olbl, Krattenthaler, and Zare~\cite{CEKZ}.

\begin{remark}\label{rem:axes-central}
    The central puncture is \emph{axes-central} if it is (approximately) equidistant from a corner puncture and the opposite
    wall, for each of the three punctures.  More specifically, suppose $A = d-a$, $B = d-b$, $C = d-c$, and $M = d-(\alpha + \beta + \gamma)$.
    There are two cases to consider:
    \begin{enumerate}
        \item If $A$, $B$, and $C$ have the same parity, then the region is of the form
            \[
                T_{A+B+C+M}(x^{B+C+M}, y^{A+C+M}, z^{A+B+M}, x^{\frac{1}{2}(B+C)} y^{\frac{1}{2}(A+C)} z^{\frac{1}{2}(A+B)}).
            \]
        \item If $A$ and $B$ differ in parity from $C$, then the region is of the form
            \[
                T_{A+B+C+M}(x^{B+C+M}, y^{A+C+M}, z^{A+B+M}, x^{\frac{1}{2}(B+C+1)} y^{\frac{1}{2}(A+C-1)} z^{\frac{1}{2}(A+B)}).
            \]
    \end{enumerate}
    \begin{figure}[!ht]
        \begin{minipage}[b]{0.48\linewidth}
            \centering
            \includegraphics[scale=1.0]{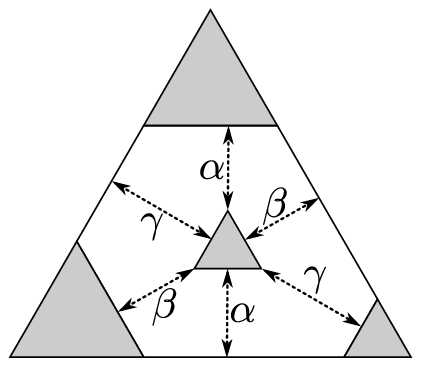}\\
            \emph{(i) The parity of $C$ agrees with $A$ and $B$.}
        \end{minipage}
        \begin{minipage}[b]{0.48\linewidth}
            \centering
            \includegraphics[scale=1.0]{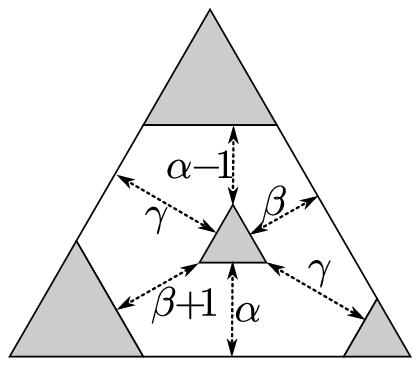}\\
            \emph{(ii) The parity of $C$ differs from $A$ and $B$.}
        \end{minipage}
        \caption{The two prototypical figures with axes-central punctures.}
        \label{fig:axes-central}
    \end{figure}

    The explicit signed enumerations for these regions can be found in~\cite[Theorems~1, 2, 4, \& 5]{CEKZ}.  However, the desired
    consequence for our use is that the signed enumeration is nonzero if and only if not all of $A$, $B$, and $C$ are odd.  Moreover,
    if it is nonzero, then the largest prime divisor of the enumeration is bounded above by $d - 1 = A+B+C+M-1$.
\end{remark}

Now, we can decide the presence of the weak Lefschetz property in almost all cases.

\begin{theorem} \label{thm:amaci-wlp}
    Let $I = I_{a,b,c,\alpha,\beta,\gamma} = (x^a, y^b, z^c, x^\alpha y^\beta z^\gamma)$ be an Artinian ideal with four
    minimal generators such that $\alpha$, $\beta$, and $\gamma$ are all positive. Assume the base field $K$ has
    characteristic zero, and consider the following conditions:
    \begin{enumerate}
        \item $\max\{a, b, c, \alpha + \beta + \gamma\} \leq d$;
        \item $\min\{\alpha + \beta + c, \alpha + b + \gamma, a + \beta + \gamma\} \geq d$;
        \item $\min\{a+b, a+c, b+c\} \geq d$; and
        \item $d = \frac{1}{3}(a+b+c+\alpha + \beta + \gamma)$ is an integer.
    \end{enumerate}
    Then the following statements hold:
        \begin{itemize}
            \item[(a)] If one of the conditions (i) - (iv) is not satisfied, then $R/I$ has the weak Lefschetz property.
            \item[(b)] Assume all the conditions (i) - (iv) are satisfied. Then:
            \begin{itemize}
                \item[(1)] The multiplication map $\times (x+y+z): [R/I]_{j-2} \to [R/I]_{j-1}$ has maximal rank whenever  $j \neq d$.
                \item[(2)] The algebra $R/I$ has the weak Lefschetz property if one of the following conditions is satisfied:
                \begin{itemize}
                    \item[(I)] Condition (ii) is an equality.
                    \item[(II)] $a+b+c+\alpha+\beta + \gamma$ is divisible by 6.
                    \item[(III)] $c = \frac{1}{2}(a+b+\alpha+\beta+\gamma)$.
                    \item[(IV)] The region $T_d(I)$ has an axes-central puncture (see Remark~\ref{rem:axes-central})
                        and one of $d-a, d-b, d-c$, and $d-(\alpha+\beta+\gamma)$ is not odd.
                    \item[(V)] $a = b$, $\alpha = \beta$, and $c$ or $\gamma$ is even.
                \end{itemize}
                \item[(3)] The algebra $R/I$ fails to have the weak Lefschetz property if one of the following
                conditions is satisfied:
                \begin{itemize}
                    \item[(IV')] The region $T_d(I)$ has an axes-central puncture (see Remark~\ref{rem:axes-central})
                        and all of $d-a, d-b, d-c$, and $d-(\alpha+\beta+\gamma)$ are odd; or
                    \item[(V')] $a = b$, $\alpha = \beta$, and both $c$ and $\gamma$ are odd.
                \end{itemize}
            \end{itemize}
    \end{itemize}
\end{theorem}
\begin{proof}
    Assertion (a) follows from Propositions~\ref{pro:amaci-semistable}, \ref{pro:amaci-nss-wlp}, and~\ref{pro:amaci-not-3}.

    Consider now the claims in part (b). Then Proposition~\ref{pro:amaci-balanced} gives that $R/I$ has the weak
    Lefschetz property if and only if $\det Z(T_d(I))$ is not zero.

    The assumptions in (b) guarantee that the punctures of $T = T_d (I)$ do not overlap and the degrees of the socle
    generators of $R/I$ are at least $d-2$. Then condition (I) implies that the puncture to the generator $x^\alpha y^\beta z^\gamma$
    touches another puncture, whereas condition (II) says that this puncture has an even side length. In either case,
    $R/I$ has the weak Lefschetz property by Proposition~\ref{prop:same-sign}.

    The proof of (b)(1) uses the Grauert-M\"ulich splitting theorem. We complete this part below
    Proposition~\ref{pro:split-type-semist}.

    The remaining assertions all follow from results in~\cite{CN-mirror-symmetry} and~\cite{CN-small-type}, when combined with
    Proposition~\ref{pro:amaci-balanced}:

    (III). The condition $c = \frac{1}{2}(a+b+\alpha+\beta+\gamma)$ is equivalent to $d-c = 0$. After taking  into account all lozenges forced by the puncture to $x^{\alpha} y^{\beta} z^{\gamma}$, the remaining subregion of $T_d (I)$ is a hexagon, and so $\det Z (T_d (I)) \neq 0$ (see, e.g., Proposition~\ref{prop:same-sign}).

    (IV) and (IV').   Use~\cite[Theorems~1, 2, 4, \& 5]{CEKZ}, as mentioned in Remark~\ref{rem:axes-central}.

    (V) and (V'). Use the results in~\cite{CN-mirror-symmetry}.
\end{proof}

Notice that Theorem~\ref{thm:amaci-wlp}(b)(1) says that, for almost monomial complete intersections, the multiplication map can fail to
have maximal rank in at most one degree.

\begin{remark}\label{rem:q-and-a}
    \begin{enumerate}
        \item Theorem~\ref{thm:amaci-wlp} can be extended to fields of sufficiently positive characteristic by using
            Proposition~\ref{pro:char-0-to-p}. This lower bound on the characteristic can be improved whenever one knows the
            determinant of $Z(T_d(I))$.
        \item Question~8.2(2c) in \cite{MMN-2011} asked if there exist non-level almost complete intersections which never
            have the weak Lefschetz property. The almost complete intersection $I = I_{3,5,5,1,2,2} = (x^3, y^5, z^5, xy^2z^2)$
            is not level and never has the weak Lefschetz property, regardless of field characteristic, as
            $\det{Z(T_6(I))} = 0$.
    \end{enumerate}
\end{remark}~

\subsection{Level almost complete intersections}\label{sub:level}~\par

In the previous subsection, we considered one way of centralizing the inner puncture of a triangular region
associated to a monomial almost complete intersection. We called such punctures ``axes-central.'' In this section, we
consider another method of centralizing the inner puncture of such a triangular region. It turns out this method of
centralization is equivalent to the algebra being level.

Consider the ideal $I = I_{a,b,c,\alpha,\beta,\gamma}$ as above. Let $d$ be an integer and assume that $T = T_d(I)$ has
one floating puncture. We say the inner puncture of $T$ is a \emph{gravity-central puncture}%
\index{puncture!gravity-central}
if the vertices of the puncture are each the same distance from the puncture opposite to it (see Figure~\ref{fig:gravity-central}).

\begin{figure}[!ht]
    \includegraphics[scale=1]{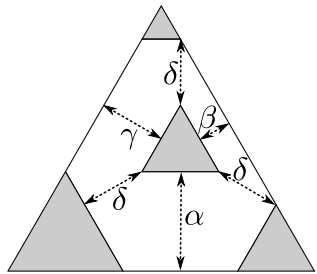}
    \caption{A prototypical figure with a gravity-central puncture.}
    \label{fig:gravity-central}
\end{figure}

\begin{lemma}
    Let $I = I_{a,b,c,\alpha,\beta,\gamma}$. Then $T_d(I)$ has a gravity-central puncture if and only if  $R/I$ is a level  algebra.
\end{lemma}
\begin{proof}
    The defining property for the distances is $(d-b) + (d-c) - \alpha = (d-a) + (d-c) - \beta = (d-a) + (d-b) - \gamma$.
    This is equivalent to the condition in Proposition~\ref{pro:amaci-props} that $R/I$ is level, i.e.,
    $a - \alpha = b - \beta = c - \gamma$.
\end{proof}

Level almost complete intersections were studied extensively in~\cite[Sections~6 and~7]{MMN-2011}. In particular,
Migliore, Mir\'o-Roig, and the second author proposed a conjectured characterization for the presence of the weak
Lefschetz property for such algebras. We recall it here, though we present it in a different, but equivalent, form to
better elucidate the reasoning behind it.

\begin{conjecture}{\cite[Conjecture~6.8]{MMN-2011}} \label{conj:level-wlp}
    Let $I = I_{\alpha+t,\beta+t, \gamma+t, \alpha,\beta,\gamma}$ be an ideal of $R = K[x,y,z]$, where $K$ has
    characteristic zero, $0 < \alpha \leq \beta \leq \gamma \leq 2(\alpha+\beta)$, $t \geq \frac{1}{3}(\alpha+\beta+\gamma)$,
    and $\alpha + \beta + \gamma$ is divisible by three. If $(\alpha,\beta,\gamma,t)$ is not $(2,9,13,9)$ or
    $(3,7,14,9)$, then $R/I$ fails to have the weak Lefschetz property if and only if $t$ is even, $\alpha + \beta + \gamma$
    is odd, and $\alpha = \beta$ or $\beta = \gamma$. Furthermore, $R/I$ fails to have the weak Lefschetz property in
    the two exceptional cases.
\end{conjecture}

The necessity part of this conjecture was proven in \cite[Corollary~7.4]{MMN-2011}) by showing that $R/I$ does not have
the weak Lefschetz property if $t$ is even, $\alpha + \beta + \gamma$ is odd, and $\alpha = \beta$ or $\beta = \gamma$.
This result is covered by Theorem~\ref{thm:amaci-wlp}(b)(3)(V') because the region is mirror symmetric. It remained open
to establish the presence of the weak Lefschetz property. Theorem~\ref{thm:amaci-wlp} does this in many new cases.

\begin{proposition} \label{pro:level-wlp}
    Consider the ideal $I = I_{\alpha+t,\beta+t, \gamma+t, \alpha,\beta,\gamma}$ as given in
    Conjecture~\ref{conj:level-wlp}. Then $R/I$ has the weak Lefschetz property if one of the following
    conditions is satisfied:
    \begin{enumerate}
        \item $t$ and $\alpha + \beta + \gamma$ have the same parity; or
        \item $t$ is odd and $\alpha = \beta = \gamma$ is even.
    \end{enumerate}
\end{proposition}
\begin{proof}
    We apply Theorem~\ref{thm:amaci-wlp} with $d = t + \frac{2}{3} (\alpha + \beta + \gamma)$. Then the side length of
    the inner puncture of $T_d(I)$ is $t - \frac{1}{3} (\alpha + \beta + \gamma)$. Hence (i) follows from
    Theorem~\ref{thm:amaci-wlp}(b)(II). Claim (ii) is a consequence of Theorem~\ref{thm:amaci-wlp}(b)(IV) as the given
    condition implies the inner puncture is axes-central.
\end{proof}

\begin{remark}
    Conjecture~\ref{conj:level-wlp} remains open in two cases, both of which are conjectured to have the weak Lefschetz property:
    \begin{enumerate}
        \item $t$ even, $\alpha + \beta + \gamma$ is odd, and $\alpha < \beta < \gamma$; and
        \item $t$ odd, $\alpha + \beta + \gamma$ is even, and $\alpha \le \beta$ or $\beta \le \gamma$.
    \end{enumerate}
\end{remark}

Notice that $T = T_d(I_{a,b,c,\alpha,\beta,\gamma})$ is simultaneously axis- and gravity-central precisely if either
$a = b = c$ and $\alpha = \beta = \gamma$, or $a = b+2 = c+1$ and $\alpha=\beta+2=\gamma+1$. In the former case, the weak
Lefschetz property in characteristic zero is completely characterized below, strengthening
\cite[Corollary~7.6]{MMN-2011}.

\begin{corollary}
    Let $I = I_{a, a, a, \alpha, \alpha, \alpha} = (x^a, y^a, z^a, x^{\alpha}, y^{\alpha}, z^{\alpha})$, where $a > \alpha$.
    Then $R/I$ fails to have the weak Lefschetz property in characteristic zero if and only if $\alpha$ and $a$ are odd
    and $a \geq 2 \alpha + 1$.
\end{corollary}
\begin{proof}
    If $a < 2 \alpha$, then $R/I$ has the weak Lefschetz property by Theorem~\ref{thm:amaci-wlp}(a).

    Assume now $a \geq 2 \alpha$. Then $R/I$ fails the weak Lefschetz property if $\alpha$ and $a$ are odd by
    \cite[Corollary~7.6]{MMN-2011} (or Theorem~\ref{thm:amaci-wlp}(b)(3)(V')). Otherwise, $R/I$ has this property by
    Proposition~\ref{pro:level-wlp}.
\end{proof}

For $a \geq 2 \alpha$, the triangular region $T_{a+\alpha} (I)$ was considered by Krattenthaler in \cite{Kr-06}. He
described a bijection between cyclically symmetric lozenge tilings of the region and descending plane partitions with
specific conditions.

~\subsection{Splitting type and regularity}~\par

The generic splitting type of a vector bundle on projective space is an important invariant. However, its computation is
often challenging. In this section we consider the splitting type of the syzygy bundles of monomial almost complete
intersections in $R$. These are rank three bundles on the projective plane. For the remainder of this section we assume
$K$ is an infinite field.

Let $I = I_{a,b,c,\alpha,\beta,\gamma}$ as above. Recall from Section~\ref{sub:syz} that the syzygy module $\syz{I}$ of
$I$ is defined by the exact sequence
\begin{equation*}
        0
    \longrightarrow
        \syz{I}
    \longrightarrow
        R(-\alpha-\beta-\gamma) \oplus R(-a) \oplus R(-b) \oplus R(-c)
    \longrightarrow
        I
    \longrightarrow
        0,
\end{equation*}
and the syzygy bundle $\widetilde\syz{I}$ on $\PP^2$ of $I$ is the sheafification of $\syz{I}$. Its restriction to any
line $H$ of $\PP^2$ splits as $\SO_H(p) \oplus \SO_H(q) \oplus \SO_H(r)$. The triple $(p, q, r)$ depends on the choice
of the line $H$, but is the same for all general lines. This latter triple is called the \emph{generic splitting type}
of $\widetilde\syz{I}$. Since $I$ is a monomial ideal, Proposition~\ref{pro:mono} implies that the
generic splitting type $(p, q, r)$ can be determined if we restrict to the line defined by $\ell = x+ y + z$.

For computing the generic splitting type of $\widetilde\syz{I}$, we use the observation that $R/(I, \ell) \cong S/J$,
where $S = K[x,y]$, and $J = (x^a, y^b, (x+y)^c, x^\alpha y^\beta (x+y)^\gamma)$. Define an $S$-module $\syz{J}$ by the
exact sequence
\begin{equation} \label{eqn:syz-J}
        0
    \longrightarrow
        \syz{J}
    \longrightarrow
        S(-\alpha-\beta-\gamma) \oplus S(-a) \oplus S(-b) \oplus S(-c)
    \longrightarrow
        J
    \longrightarrow
        0
\end{equation}
using the, possibly non-minimal, set of generators $\{x^a, y^b, (x+y)^c, x^\alpha y^\beta (x+y)^\gamma\}$ of $J$. Then
$\syz{J} \cong S(p) \oplus S(q) \oplus S(r)$, where $(p, q, r)$ is the generic splitting type of the vector bundle
$\widetilde\syz{I}$. The Castelnuovo-Mumford regularity of the ideal $J$ is $\reg{J}= 1 + \reg S/J$.

For later use we record the following facts.

\begin{remark} \label{rem:splitting-type}
    Adopt the above notation. Then the following statements hold:
    \begin{enumerate}
        \item Using, for example,  the Sequence~(\ref{eqn:syz-J}), one gets  $-(p + q + r) = a+b+c+\alpha+\beta+\gamma$.
        \item If any of the generators of $J$ is extraneous, then the degree of that generator is one of $-p$, $-q$, or $-r$.
        \item As the regularity of $J$ is determined by the Betti numbers of $S/J$, we obtain that
            $\reg{J} + 1 = \max\{-p,-q,-r\}$ if the Sequence~(\ref{eqn:syz-J}) is a minimal free resolution of $J$.
    \end{enumerate}
\end{remark}

Before moving on, we prove a technical but useful lemma.

\begin{lemma} \label{lem:reg-2AMACI}
    Let $S = K[x,y]$, where $K$ is a field of characteristic zero. Consider the ideal $\fa = (x^a, y^b, x^\alpha y^\beta (x+y)^\gamma)$
    of $S$, and assume that the given generating set is minimal. Then $\reg{\fa}$ is
    \[
            -1 + \max \left \{a+ \beta, b+\alpha, \min \left \{a+b, a+ \beta + \gamma, b+ \alpha + \gamma,
                \left\lceil \frac{1}{2}(a+b+ \alpha + \beta + \gamma)\right\rceil  \right \} \right\}.
    \]
\end{lemma}
\begin{proof}
    We proceed in three steps.

    First, considering the minimal free resolution of the ideal $(x^a, y^b, x^\alpha y^\beta)$, we conclude
    \[
        \reg (x^a, y^b, x^\alpha y^\beta) = -1 + \max \{a+ \beta, b+\alpha\}.
    \]

    Second, the algebra $S/(x^a, y^b)$ has the strong Lefschetz property in characteristic zero (see, e.g.,
    \cite[Proposition~4.4]{HMNW}). Thus, the Hilbert function of $S/(x^a, y^b, (x+y)^\gamma)$ is
    \[
        \dim_K{[S/(x^a, y^b, (x+y)^\gamma)]_j} = \max\{0, \dim_K{[S/(x^a, y^b)]_j} - \dim_K{[S/(x^a,y^b)]_{j-\gamma}}\}.
    \]
    By analyzing when the difference becomes non-positive, we get that
    \begin{equation}\label{eq:reg-restr-ci}
        \reg (x^a, y^b, (x+y)^\gamma) = -1 + \min \left \{a+b, a+ \gamma, b+\gamma, \left\lceil \frac{1}{2}(a+b+\gamma)\right\rceil \right \}.
    \end{equation}

    Third, notice that
    \[
        (x^a, y^b, x^\alpha y^\beta (x+y)^\gamma):x^\alpha y^\beta = (x^{a-\alpha}, y^{b-\beta}, (x+y)^\gamma).
    \]
    Hence, multiplication by $x^\alpha y^\beta$ induces the short exact sequence
    \[
        0 \rightarrow [S/(x^{a-\alpha}, y^{b-\beta}, (x+y)^\gamma)](-\alpha-\beta) \stackrel{\times x^\alpha y^\beta}{\longrightarrow}
        S/\fa \rightarrow S/(x^a, y^b, x^\alpha y^\beta) \rightarrow 0.
    \]
    It implies
    \[
        \reg{\fa} = \max\{\alpha + \beta + \reg{(x^{a-\alpha}, y^{b-\beta}, (x+y)^\gamma)},
        \reg{(x^a, y^b, x^\alpha y^\beta)}\}.
    \]

    Using the first two steps, the claim follows.
\end{proof}

Recall that Proposition~\ref{pro:amaci-semistable} gives a characterization of the semistability of the syzygy bundle
$\widetilde\syz{I_{a,b,c,\alpha,\beta,\gamma}}$, using only the parameters $a$, $b$, $c$, $\alpha$, $\beta$, and
$\gamma$. We determine the splitting type of $\widetilde\syz{I_{a,b,c,\alpha,\beta,\gamma}}$ for the nonsemistable and
the semistable cases separately.

~\subsubsection{Nonsemistable syzygy bundle}~

We first consider the case when the syzygy bundle is not semistable, and therein we distinguish four cases. It turns out
that in three cases, at least one of the generators of the ideal $J$ is extraneous.

\begin{proposition} \label{pro:st-nss}
    Consider the ideal $I = I_{a,b,c,\alpha,\beta,\gamma} = (x^a, y^b, z^c, x^\alpha y^\beta z^\gamma)$ with four
    minimal generators. Assume that the base field $K$ has characteristic zero and, without loss of generality, that
    $a \leq b \leq c$. Set $d := \frac{1}{3}(a+b+c+\alpha+\beta+\gamma)$, and denote by $(p, q, r)$ the generic splitting
    type of $\widetilde\syz{I}$. Assume that $\widetilde\syz{I}$ is not semistable. Then:
    \begin{enumerate}
        \item If $\min \{\alpha + \beta + \gamma, c\} \geq a+b -1$, then
            \[
                (p, q, r) = (-c, -\alpha - \beta - \gamma, -a-b).
            \]
        \item Assume $\min \{\alpha + \beta + \gamma, c\} \leq a+b -2$ and
            \[
                \frac{1}{2}(a+b+ c) \leq \min \left \{a+ \beta + \gamma, b+ \alpha + \gamma, c + \beta + \gamma,
                    \frac{1}{2}(a+b+ \alpha + \beta + \gamma) \right \}.
            \]
            Then
            \[
                (p, q, r) = (-\alpha-\beta-\gamma, - \left\lceil \frac{1}{2}(a+b+c) \right\rceil,
                    - \left\lfloor \frac{1}{2}(a+b+c) \right\rfloor).
            \]
        \item Assume $\min \{\alpha + \beta + \gamma, c\} \leq a+b -2$ and
            \[
                \frac{1}{2}(a+b+ \alpha + \beta + \gamma)  \leq \min \left \{a+ \beta + \gamma, b+ \alpha + \gamma,
                    c + \beta + \gamma, \frac{1}{2}(a+b+ c)  \right \}.
            \]
            Then
            \[
                (p, q, r) =  (-c, q, -a-b-\alpha-\beta-\gamma+q),
            \]
            where $- q = \min \left \{a+ \beta + \gamma, b+ \alpha + \gamma,
            \left\lceil \frac{1}{2}(a+b+ \alpha + \beta + \gamma)\right\rceil  \right \}$.
        \item Assume $\min \{\alpha + \beta + \gamma, c\} \leq a+b -2$ and
            \begin{equation*}
                \begin{split}
                    -s = \min \left \{a+ \beta + \gamma, b+ \alpha + \gamma, c + \beta + \gamma  \right \} < \hspace*{5cm} \\
                    \min \left \{ \frac{1}{2}(a+b+ \alpha + \beta + \gamma), \frac{1}{2}(a+b+ c) \right \}.
                \end{split}
            \end{equation*}
            Then
            \begin{equation*}
                (p, q, r) = \left (  \left\lfloor \frac{1}{2}(-3d-s) \right\rfloor,  \left\lceil \frac{1}{2}(-3 d - s) \right\rceil, s \right).
            \end{equation*}
    \end{enumerate}
\end{proposition}
\begin{proof}
    Set
    \begin{equation*}
        \mu = \min \left \{a+b, a+ \beta + \gamma, b+ \alpha + \gamma, c + \beta + \gamma,   \frac{1}{2}(a+b+ \alpha + \beta + \gamma), \frac{1}{2}(a+b+ c) \right \}.
    \end{equation*}

    Using $a \leq b \leq c$, \cite[Theorem 6.3]{Br} implies that the maximal slope of a subsheaf of $\widetilde\syz{I}$
    is $-\mu$. Since $\widetilde\syz{I}$ is not semistable, we have $\mu < d$ (see
    Proposition~\ref{pro:amaci-semistable}). Moreover, the generic splitting type of $\widetilde\syz{I}$ is determined
    by the minimal free resolution of $J = (x^a, y^b, (x+y)^c, x^\alpha y^\beta (x+y)^\gamma)$ as a module over $S = K[x,
    y]$. We combine both approaches to determine the generic splitting type.

    Since $\reg (x^a, y^b) = a+b-1$, all polynomials in $S$ whose degree is at least $a+b-1$ are contained in $(x^a,
    y^b)$. Hence, $J = (x^a, y^b)$ if $\min \{\alpha + \beta + \gamma, c\} \geq a+b -1$, and the claim in case (i)
    follows by Remark~\ref{rem:splitting-type}.

    For the remainder of the proof, assume $\min \{\alpha + \beta + \gamma, c\} \leq a+b -2$. Then $a+b > \frac{1}{2}(a+b+ c)$,
    and thus $\mu \neq a+b$.

    In case (ii), it follows that $\mu = \frac{1}{2}(a+b+ c)$ and $c \leq \alpha + \beta + \gamma$, and thus $c \leq a+b-2$.
    Using Equation \eqref{eq:reg-restr-ci}, we conclude that
    \[
        \reg (x^a, y^b, (x+y)^c) = -1 + \min \left \{a+b,  \left\lceil \frac{1}{2}(a+b+c)\right\rceil \right \} = -1 + \left\lceil \frac{1}{2}(a+b+c)\right\rceil.
    \]
    Observe now that $d > \mu = \frac{1}{2}(a+b+ c)$ is equivalent to $\alpha + \beta + \gamma > \frac{1}{2}(a+b+ c)$.
    This implies $\alpha + \beta + \gamma > \reg (x^a, y^b, (x+y)^c)$, and thus $J = (x^a, y^b, (x+y)^c)$. Using
    Remark~\ref{rem:splitting-type} again, we get the generic splitting type of $\widetilde\syz{I}$ as claimed in (ii).

    Consider now case (iii). Then $d > \mu = \frac{1}{2}(a+b+ \alpha + \beta + \gamma)$, which gives $c >
    \frac{1}{2}(a+b+ \alpha + \beta + \gamma)$. The second assumption in this case also implies $\frac{1}{2}(a+b+ \alpha
    + \beta + \gamma) \leq a + \beta + \gamma$, which is equivalent to $b+\alpha \leq a + \beta + \gamma$ and also to $b +
    \alpha \leq \frac{1}{2}(a+b+ \alpha + \beta + \gamma)$. Similarly, we have that $\frac{1}{2}(a+b+ \alpha + \beta +
    \gamma) \leq b+\alpha + \gamma$, which is equivalent to $a + \beta \leq b+ \alpha + \gamma$ and also to $a + \beta \le
    \frac{1}{2}(a+b+ \alpha + \beta + \gamma)$. It follows that
    \[
        \max \{a+\beta, b+ \alpha \} \leq  \min \left \{a+ \beta + \gamma, b+ \alpha + \gamma, \frac{1}{2}(a+b+ \alpha + \beta + \gamma)  \right \}.
    \]
    Hence Lemma~\ref{lem:reg-2AMACI} yields
    \begin{equation*}
        \begin{split}
            \reg (x^a, y^b,  x^\alpha y^\beta (x+y)^\gamma) =   \hspace*{9.7cm} \\
            -1 +   \min \left \{a+ \beta + \gamma, b+ \alpha + \gamma, \left\lceil \frac{1}{2}(a+b+ \alpha + \beta + \gamma)\right\rceil  \right \} < c.
        \end{split}
    \end{equation*}
    This shows that $(x+y)^c \in (x^a, y^b, x^\alpha y^\beta (x+y)^\gamma) = J$. Setting $- q = 1 + \reg J$,
    Remark~\ref{rem:splitting-type} provides the generic splitting type in case (iii).

    Finally consider case (iv). Then $\mu = -s$, and $\mu$ is equal to the degree of the least common multiple of two of
    the minimal generators of $I$. In fact, $-\mu = s$ is the slope of the syzygy bundle ${\mathcal O}_{\PP^2}(s)$ of
    the ideal generated by these two generators. Thus, the Harder-Narasimhan filtration (see
    \cite[Definition~1.3.2]{HM}) gives an exact sequence
    \[
        0 \to {\mathcal O}_{\PP^2}(s) \to \widetilde\syz{I} \to {\mathcal E} \to 0,
    \]
    where ${\mathcal E}$ is a semistable torsion-free sheaf on $\PP^2$ of rank two and first Chern class $-a-b-c -
    \alpha - \beta - \gamma -s = -3d -s$. Its bidual ${\mathcal E}^{**}$ is a stable vector bundle. Thus, by the theorem
    of Grauert and M\"ulich (see \cite{GM} or \cite[Corollary 1 of Theorem 2.1.4]{OSS}), its generic splitting type is
    $( \left\lfloor \frac{1}{2}(-3d-s) \right\rfloor, \left\lceil \frac{1}{2}(-3 d - s) \right\rceil)$. Now the claim
    follows by restricting the above sequence to a general line of $\PP^2$.
\end{proof}

We have seen that the ideal $J = (x^a, y^b, (x+y)^c, x^\alpha y^\beta (x+y)^\gamma)$ has at most three minimal
generators in the cases (i) - (iii) of the above proposition. In the fourth case, the associated ideal $J \subset S$ may
be minimally generated by four polynomials.

\begin{example} \label{exa:st-nss-4mingen}
    Consider the ideal
    \[
        I = I_{4,5,5,3,1,1} = (x^4, y^5, z^5, x^3yz).
    \]
    Then the corresponding ideal $J$ is minimally generated by $x^4, y^5, (x+y)^5$, and $x^3y(x+y)$. The syzygy bundle
    of $\widetilde\syz{I}$ is not semistable, and its generic splitting type is $(-7, -6, -6)$ by
    Proposition~\ref{pro:st-nss}(iv).
\end{example}~

\subsubsection{Semistable syzygy bundle}~

Order the entries of the generic splitting type $(p,q,r)$ of the semistable syzygy bundle $\widetilde\syz{I}$ such
that $p \leq q \leq r$. In this case, the splitting type determines the presence of the weak Lefschetz property if
the characteristic of $K$ is zero (see \cite[Theorem 2.2]{BK}). The following result is slightly more precise.

\begin{proposition}\label{pro:split-type-semist}
    Let $K$ be a field of characteristic zero, and assume the ideal $I = I_{a,b,c,\alpha,\beta,\gamma}$ has a semistable
    syzygy bundle. Set $k = \left\lfloor \frac{1}{3}(a+b+c+\alpha+\beta+\gamma) \right\rfloor$. Then the generic
    splitting type of $\widetilde\syz{I}$ is
    \begin{equation*}
        (p, q, r) =
        \begin{cases}
            (-k-1,-k,-k) & \text{if } a+b+c+\alpha+\beta+\gamma = 3k+1;\\
            (-k-1,-k-1,-k) & \text{if } a+b+c+\alpha+\beta+\gamma = 3k+2; \\
            (-k,-k,-k) & \text{if } a+b+c+\alpha+\beta+\gamma = 3k \text{ and} \\
                       & \text{$R/I$ has the weak Lefschetz property}; \\
            (-k-1,-k,-k+1) & \text{if } a+b+c+\alpha+\beta+\gamma = 3k \text{ and} \\
                       & \text{$R/I$ fails to have the weak Lefschetz property}.
        \end{cases}
    \end{equation*}
\end{proposition}
\begin{proof}
    The Grauert-M\"ulich theorem \cite{GM} gives that $r - q$ and $q - p$ are both nonnegative and at most 1. Moreover,
    $p, q$, and $r$ satisfy $a+b+c+\alpha+\beta+\gamma = -(p+q+r)$ (see Remark~\ref{rem:splitting-type}(i)). This gives
    the result if $k \neq d = \frac{1}{3}(a+b+c+\alpha+\beta+\gamma)$.

    It remains to consider the case when $k = d$. Then $(-k,-k,-k) $ and $(-k-1,-k,-k+1)$ are the only possible generic
    splitting types. By Proposition~\ref{pro:amaci-semistable}(i), the minimal generators of the ideal
    $J = (x^a, y^b, (x+y)^c, x^\alpha y^\beta (x+y)^\gamma)$ have degrees that are less than $d$. Hence $\reg J = d$ if and
    only if the splitting type of $\widetilde\syz{I}$ is $(-d-1, -d , -d+1)$. Since $\dim_K [R/I]_{d-2} = \dim_K [R/I]_{d-1}$,
    using Proposition~\ref{pro:amaci-balanced}, we conclude that $\reg J \geq d$ if and only if $R/I$ does not have the
    weak Lefschetz property.
\end{proof}

We are ready to add the missing piece in the proof of Theorem~\ref{thm:amaci-wlp}.

\begin{proof}[Completion of the proof of Theorem~\ref{thm:amaci-wlp}(b)(1)]
    \mbox{ }

    We have just seen that the ideal $J = (x^a, y^b, (x+y)^c, x^\alpha y^\beta (x+y)^\gamma)$ has regularity $d$ if
    $R/I$ fails the weak Lefschetz property. This implies that the multiplication map $\times (x+y+x): [R/I]_{j-2} \to [R/I]_{j-1}$
    is surjective whenever $j > d$. Moreover, since the minimal generators of $J$ have degrees that are less than $d$,
    we have the exact sequence
    \begin{equation*}
            0
        \longrightarrow
            S(-d+1) \oplus S(-d) \oplus S(-d-1)
        \longrightarrow
            S(-\alpha-\beta-\gamma) \oplus S(-a) \oplus S(-b) \oplus S(-c)
        \longrightarrow
            J
        \longrightarrow
            0.
    \end{equation*}

    In the above proof of Theorem~\ref{thm:amaci-wlp} we saw that the four punctures of $T_d (I)$ do not overlap and
    that $T_d(I)$ is balanced. Hence $T_{d-1} (I)$ has 3 more downward-pointing than upward-pointing triangles, that is,
    \[
        \dim_K [R/I]_{d-2} = \dim_K [R/I]_{d-3} + 3.
    \]
    It follows that the multiplication map in the exact sequence
    \[
        [R/I]_{d-3} \longrightarrow [R/I]_{d-2} \longrightarrow S/J \longrightarrow 0
    \]
    is injective because $\dim_K [S/J]_{d-2} = 3$. Hence $\times (x+y+x): [R/I]_{j-2} \to [R/I]_{j-1}$ is injective whenever $j \leq d-1$.
\end{proof}

The second author would like to thank the authors of \cite{MMMNW}; it was during a conversation in the preparation of that paper that he
learned about the use of the Grauert-M\"ulich theorem for an alternative way of deducing the injectivity of the map
$[R/I]_{d-3} \longrightarrow [R/I]_{d-2}$ in the above argument if the characteristic of $K$ is zero.

\begin{example} \label{exa:syzygy}
    Consider the ideal $I_{7,7,7,3,3,3} = (x^7, y^7, z^7, x^3 y^3 z^3)$. It never has the weak Lefschetz property, by
    Theorem~\ref{thm:amaci-wlp}(vii). The bundle $\widetilde\syz{I_{7,7,7,3,3,3}}$ has
     generic splitting type  $(-11, -10, -9)$.
    Notice that the similar ideal $I_{6,7,8,3,3,3} = (x^6, y^7, z^8, x^3 y^3 z^3)$ has the weak Lefschetz property
    in characteristic zero as $\det{N_{6,7,8,3,3,3}} = -1764$. The generic splitting type of
    $\widetilde\syz{I_{6,7,8,3,3,3}}$ is $(-10,-10, -10)$.
\end{example}

We summarize part of our results for the case where $I$ is associated to a tileable triangular region.
In particular, if $K$ is an infinite field of arbitrary characteristic, then the splitting type can be
used to determine the presence of the weak Lefschetz property.

\begin{theorem} \label{thm:equiv}
    Let $I = I_{a,b,c,\alpha,\beta,\gamma} \subset R = K[x,y,z]$, where $K$ is an infinite field of arbitrary
    characteristic. Assume $I$ satisfies conditions (i)--(iv) in Theorem~\ref{thm:amaci-wlp} and
    $d := \frac{1}{3}(a+b+c+\alpha+\beta+\gamma)$ is an integer.
    Then the following conditions are equivalent:
    \begin{enumerate}
        \item The algebra $R/I$ has the weak Lefschetz property.
        \item The determinant of $Z(T_d(I))$ (i.e., the enumeration of signed perfect matchings of the bipartite graph $G(T_d(I)$) is not
            zero in $K$.
        \item The generic splitting type of $\widetilde\syz{I}$ is $(-d,-d,-d)$.
    \end{enumerate}
\end{theorem}
\begin{proof}
    Regardless of the characteristic of $K$, the arguments for Proposition~\ref{pro:amaci-balanced} show that $T_d(I)$
    is balanced. Moreover, the degrees of the socle generators of $R/I$ are at least $d-2$ as shown in
    Theorem~\ref{thm:amaci-wlp}(b)(1). Hence, Proposition~\ref{pro:wlp} gives that $R/I$ has the weak
    Lefschetz property if and only if the multiplication map
    \[
        \times (x+y+z): [R/I]_{d-2} \to [R/I]_{d-1}
    \]
    is bijective. Now, Corollary~\ref{cor:wlp-biadj} yields the equivalence of Conditions (i) and (ii).

    As above, let $(p, q, r)$ be the generic splitting type of $\widetilde\syz{I}$, where $p \leq q \leq r$, and let
    $J \subset S$ be the ideal such that $R/(I, x+y+z) \cong S/J$. The above multiplication map is bijective if and only if
    $\reg J = d-1$. Since $\reg J + 1 = -r$ and $p+q+r = -3d$, it follows that $\reg J = d-1$ if and only if $(p, q, r)
    = (-d, -d, -d)$. Hence, conditions (i) and (iii) are equivalent.
\end{proof}


\end{document}